\documentclass{amsart}
\usepackage{amssymb}
\usepackage{amsthm}
\usepackage{nccmath}
\usepackage{tikz-cd}
\usepackage[normalem]{ulem}
\usepackage{hyperref}
\usepackage{parskip}

\usepackage[curve,matrix,arrow]{xy}
\theoremstyle{plain}\newtheorem{Theorem}{Theorem}[section]
\theoremstyle{plain}
\theoremstyle{plain}\newtheorem{Corollary}[Theorem]{Corollary}
\theoremstyle{plain}\newtheorem{Lemma}[Theorem]{Lemma}
\theoremstyle{plain}\newtheorem{Proposition}[Theorem]{Proposition}
\theoremstyle{plain}
\theoremstyle{definition}\newtheorem{Definition}[Theorem]{Definition}
\theoremstyle{definition}
\theoremstyle{definition}
\theoremstyle{definition}
\theoremstyle{definition}\newtheorem{Remark}[Theorem]{Remark}
\theoremstyle{definition}

\def\cL{{\mathcal{L}}}

\def\cB{{\mathcal{B}}}
\def\cR{{\mathcal{R}}}

\def\cI{{\mathcal{I}}}
\def\cN{{\mathcal{N}}}

\def\ker{{\rm{Ker}}}

\def\HH{{\rm{HH}}}

\usepackage{color}
\definecolor{tealgreen}{rgb}{0.0, 0.71, 0.5}
\title[The first Hochschild cohomology as a Lie algebra]{The first Hochschild cohomology as a Lie algebra}

\author[Rubio y Degrassi]{Lleonard Rubio y Degrassi}
\address{
Departamento de Matem\'aticas \\
Universidad de Murcia \\
30100 Murcia\\
Spain
}
\email{lleonard.rubio@um.es}

\author[Schroll]{Sibylle Schroll}
\address{
Department of Mathematics \\
University of Leicester \\
University Road  \\
Leicester LE1 7RH, UK
}
\email{schroll@le.ac.uk}

\author[Solotar]{Andrea Solotar}
\address{
Departamento de Matem\'atica \\
 FCEyN Universidad de Buenos Aires \\ 
 Pabellon I - Ciudad Universitaria \\
1428 - Buenos Aires \\ 
  Argentina}
\email{asolotar@dm.uba.ar}
 \date{}

\thanks{\noindent\smallskip
{\it Acknowledgements.} The
first  author has been supported by the Oberwolfach Leibniz Fellows Program, the DAAD Short-Term Grant 57378443, Fundaci\'on S\'eneca of Murcia
19880/GERM/15 and INdAM postdoctoral grant.  The second author is supported by EPSRC through an Early Career Fellowship EP/P016294/1.  The third author, a research member of CONICET (Argentina) and a Senior Associate Professor ICTIP, has been partially supported by the projects UBACYT 20020160100613BA and PIP-CONICET 112201501
00483CO. The first and the third author thank the University of Leicester for their  hospitality. }
\begin{document}
\begin{abstract}
In this paper we study sufficient conditions for the solvability of the first Hochschild cohomology of a finite dimensional algebra as a Lie 
algebra in terms of its Ext-quiver in arbitrary characteristic. In particular, we show that if the quiver has no parallel arrows and no 
loops then the first Hochschild cohomology is solvable. For quivers containing loops, we determine easily verifiable sufficient conditions
for the solvability of the first Hochschild cohomology.
We apply these criteria to show the solvability of the first Hochschild cohomology space
for  large families of algebras, namely,  several families of self-injective tame algebras  including all tame blocks of finite groups and 
some wild algebras including most quantum 
complete intersections.
\end{abstract}

\raggedbottom
\maketitle


\section{Introduction}
Let $K$ be an algebraically closed field and let $A$ be a
finite dimensional $K$-algebra.
The problem of describing the Hochschild cohomology of $A$ as a Gerstenhaber algebra and in particular the first Hochschild cohomology space 
as a Lie algebra, and how this structure is related to $A$, has been studied in several recent articles, see for example \cite{ALS, 
CSS, ER, MNPRS, SA}. 
The Gerstenhaber bracket in Hochschild cohomology has been defined more than fifty years ago. Recently,  new methods to explicitly compute it
have been introduced rendering the problem more tractable. An  interesting question arising  is which 
Lie algebras can actually appear in this way. The results in this article contribute to an answer to  this question.  
Namely, we give sufficient criteria for the solvability as a Lie algebra of the first degree Hochschild  cohomology of $A$, denoted 
by $\mathrm{HH}^1(A)$, in terms of its Ext-quiver and its relations.


The Hochschild cohomology  $\mathrm{HH}^{\bullet}(A)=\bigoplus_{n\geq 0}\mathrm{HH}^n(A)$  of  $A$ 
has a very rich structure. It is an associative, graded-commutative algebra with the cup product. It also has a graded Lie bracket 
of  degree $-1$ and both structures are related by the graded Poisson identity. 
In particular, $\mathrm{HH}^{\bullet}(A)$ is a Gerstenhaber algebra, $\mathrm{HH}^{1}(A)$ is a Lie algebra with bracket induced 
by the usual commutator of derivations and $\mathrm{HH}^{\bullet}(A)$ is a Lie module for this Lie algebra.
All these structures are invariant under derived equivalence \cite{K}.

In this paper we study the solvability of the Lie algebra given by the first Hochschild cohomology of a finite dimensional algebra. We develop sufficient conditions on the Ext-quiver with relations so that this Lie algebra is solvable and we give several applications of our methods. In particular, we show that for wild algebras the first Hochschild cohomology  is sometimes  solvable and sometimes semi-simple. We show that it is solvable for most quantum complete intersections for arbitrary parameter and semi-simple for a family of algebras related to Beilinson algebras. 

The first condition for solvability is based on the Ext-quiver having no parallel arrows and no loops. 

\begin{Theorem}  (see Theorem~\ref{Liestronglysolvable-general})
\label{parallel}
Let $K$ be an algebraically closed field and let $A$ be a  finite
dimensional 
$K$-algebra such that \sloppy $\mathrm{dim}_K(\mathrm{Ext}^1_A(S,S))= 0$ and $\mathrm{dim}_K(\mathrm{Ext}^1_A(S,T))\leq1$, 
for all non-isomorphic simple $A$-modules $S$ and $ T$. Then $\mathrm{HH}^1 (A)$ is a strongly solvable Lie algebra. 

In particular,  $\mathrm{HH}^1 (A)$ is a solvable Lie algebra. 
\end{Theorem}

The solvability of the first Hochschild cohomology has recently been extensively studied \cite{CSS, ER, LR} and Theorem~\ref{parallel} also appears in \cite{LR} with a different proof. 
In this paper,  we combine results from \cite{ER} with  a version for graded algebras of Theorem 1.1, to obtain Theorem 1.1. Our next main result concerns algebras whose Ext-quiver may have loops. 

\begin{Theorem}(see Theorem~\ref{trivialclasses-nongraded})
\label{trivialclasses-nongraded-introduction}
Let $K$ be an algebraically closed field and let $A$  be a finite dimensional  
$K$-algebra 
such that \sloppy
$\rm{dim_K Ext}^1_A(S,T) \leq1$ for  simple $A$-modules $S$ and $T$. Suppose that  there is no derivation of $A$ sending a loop in the Ext-quiver to the identity. Then $\mathrm{HH}^1(A)$ is a solvable Lie algebra. 
\end{Theorem}

We then give some further criteria for the solvability of $\HH^1$ for algebras whose Ext-quiver has at most one loop at each vertex if there is more than one vertex and at most two loops 
in the local case.

Furthermore, we show some criteria for the solvability of the first Hochschild cohomology of graded algebras and we apply these to show that the first Hochschild cohomology of most quantum complete intersections is  strongly solvable.

In \cite{Skow} an extensive if not comprehensive classification of self-injective tame algebras up to derived equivalence has been given. We consider all families of 
symmetric algebras in this classification and we show that for almost all these algebras    the first Hochschild cohomology is solvable in arbitrary characteristic. We also show that almost all tame blocks of finite groups  have solvable first Hochschild cohomology.

\begin{Theorem} 
\label{thm-Skowronski} (see Theorem~\ref{selfinjetivetame}) 
 Let $K$ be an algebraically closed field of arbitrary characteristic. 
 Let $A$ be a symmetric finite dimensional  $K$-algebra  of tame representation type appearing in the classification in \cite{Skow} which is  not derived equivalent to  $K[X]/(X^r)$ 
when $\mathrm{char}(K)\mid r$,  the trivial extension of the Kronecker algebra if $\rm{char}(K) \neq 2$  or  $K[X,Y]/(X^2, Y^2)$  if the characteristic is $2$. 
Then $\HH^1(A)$ is a  solvable Lie algebra. 

In particular, let $A$ be a symmetric tame algebra of dihedral, semi-dihedral or quaternion type not derived equivalent to $K[X,Y]/(X^2, Y^2)$  if the characteristic is $2$. Then $\mathrm{HH^1}(A)$ is a solvable Lie algebra.
\end{Theorem}

 We note that there  are two exceptions for the solvability of $\HH^1(A)$  in the first part of the theorem  and one in the second part. Namely,   
 if char$(K)|r$ and  $r \geq3$, the Lie algebra $\HH^1(K[X]/(X^{r}))$ is perfect. In particular, if $r = p = \rm{char}(K)$ then 
 it is simple and it is the so-called  Jacobson-Witt algebra.  Similarly,  $\HH^1(K[X,Y]/(X^2, Y^2))$ in characteristic 2 is also a Jacobson--Witt algebra and not solvable.  The second exception in the first part is if $A$ is   derived equivalent to the trivial extension of the Kronecker algebra 
 and $\rm{char}(K) \neq 2$, in which case the obtained Lie algebra is isomorphic to
$\mathfrak{sl}_2(K)$, see \cite{CSS}.

 We also show that the first Hochschild cohomology of a Brauer graph algebra with any multiplicity function is solvable in arbitrary characteristic, that is with the exception of the trivial extension of the Kronecker quiver in characteristic different from 2 and some other cases in characteristic 2.

The article is structured as follows. Section~\ref{Background} contains background material on Lie algebras both in characteristic zero and in positive characteristic. It also contains a brief introduction of the Lie algebra structure of the first Hochschild cohomology of a finite dimensional algebra. Section~\ref{conditions} contains the main results of the paper, that is criteria for the solvability of the first Hochschild cohomology of a finite dimensional algebra in terms of its Ext-quiver and the relations on the quiver.   In Section~\ref{applications}, we apply the results of Section~\ref{conditions} to several families of algebras such  as symmetric   algebras of tame representation type including Brauer graph algebras and blocks of groups algebras of finite groups as well as quantum complete intersections. We end the paper with an example of an infinite family of algebras for which the first Hochschild cohomology is semi-simple.

{\bf Acknowledgments.} We thank Cristian Chaparro and Mariano Su\'arez \'Alvarez for helpful discussions. We thank the referee for the very careful reading of the paper and for many helpful suggestions and for pointing out a consequential oversight in a previous version of the paper. The third and the first author thank the University of Leicester for its hospitality during their stays.

\section*{Conventions}

Let $K$ be an algebraically closed field. For a finite dimensional  $K$-algebra $A$, the Ext-quiver of $A$ is the quiver  whose vertices are in bijection with the isomorphism classes of simple $A$-modules and where the number of arrows from a vertex corresponding to a simple module $S$ to a simple module $T$ is given by ${\rm dim}_K({\rm Ext^1}_A(S, T))$. Furthermore, if  $A=KQ/I$ is a finite dimensional $K$-algebra, with quiver $Q = (Q_1, Q_0)$ with $Q_0$ the vertices  of $Q$ and $Q_1$ the arrows in $Q$, then unless otherwise stated, we assume that the ideal $I$ is admissible. For an arrow $a$ in $Q$, 
we write $s(a)$ for the source of $a$ and $t(a)$ for the target of $a$.    
We assume that  all modules will be right modules and we write $A^e = A \otimes_K A^{op}$.
 We denote by $J(A)$ the Jacobson radical of $A$.


\section{Background material}
\label{Background}

In this section we collect some background material on modular Lie algebras and the Lie bracket on the first Hochschild cohomology of a finite dimensional 
algebra (in any characteristic). For this let $\cL$ be a Lie algebra and recall that the derived Lie algebra $\cL^{(1)}$ is the Lie algebra defined by the linear span of all commutators $[x,y]$ for $x, y \in \cL$. We denote by
$\cL^{(i)} = [\cL^{(i-1)}, \cL^{(i-1)}]$ the $i$-th term of the derived series of $\cL$,  where $i \geq 1$ and $\cL^{(0)} = \cL$.

\subsection{Modular Lie algebras} 
We begin by collecting  some well-known facts about modular Lie algebras. 

\begin{Definition}
A Lie algebra $\mathcal{L}$ is {\it strongly solvable} if its derived subalgebra   $\cL^{(1)}$ is nilpotent. 
\end{Definition}

\begin{Remark}
Note that any strongly solvable Lie algebra is solvable but the converse only holds in characteristic zero. 
\end{Remark}

Furthermore, just as in the case of solvable Lie algebras, we have the following. 

\begin{Lemma}
\label{SuperSol}
Every subalgebra and every factor of a strongly solvable Lie algebra is strongly solvable.
\end{Lemma}

\subsection{The Lie bracket on the first Hochschild cohomology}
\label{Brackets}

For a finite dimensional $K$-algebra $A = KQ/I$, we briefly recall the construction of the Lie bracket on $HH^1(A)$ as defined in  \cite{CS}. 
We have $A = E \oplus J(A)$ where $E=KQ_0$ is  a maximal 
semi-simple subalgebra of $A$ and $J(A)$ is the Jacobson radical of $A$. 
The set of paths is a basis of the path algebra $KQ$. When we consider the quotient $KQ/I$ this is not a basis anymore, but we can extract from the 
set of paths a basis consisting of some of these paths.
In order to do so we need to fix a reduction system. For this, 
let $p, q  $ be paths in  $Q$. We say that $p$ \textit{divides} $q$ if there exist paths $a,b$ in $Q$ such that 
$q = apb$. A set  $R = \{(s_i, f_i) \in Q_{\geq 0}  \times KQ\}$ is called a \textit{reduction system} if for all $(s, f)\in R$, $s$ is parallel to $f$, 
$s \neq f$ and for any $(s', f') \in R$  we have that  $s^\prime$ does not divide $s$ if $s^\prime \neq s$.  
 
Let $\leq$ be a well-order on the set $Q_0 \cup Q_1$ such that $e < \alpha$, for all $e \in Q_0$ and 
$\alpha \in Q_1$. Let $\omega: Q_1 \to \mathbb{N}  \cup \{0 \}$ be a function which we extend  to $Q_{\geq 0}$
by imposing that $\omega(e) = 0$, for all $e \in Q_0$ and $\omega(c_n \dots c_1) = \sum_{i=1}^n
\omega (c_i)$, for $c_i \in Q_1$. Let $c, d \in Q_{\geq 0}$. We write
that $c \leq_{\omega} d$ if one of the following conditions hold:
\begin{itemize}
\item $c, d \in Q_0$ and $c \leq d$, 
\item $\omega(c) < \omega(d)$, 
\item $\omega(c) = \omega(d)$, $c = c_n \dots c_1$, $d = d_m \dots d_1 \in Q_{\geq 1}$ and there exists 
$j \leq min(length(c), length(d))$ such that $c_i = d_i$ for all $1\leq i\leq j -1$ and $c_j < d_j$.
\end{itemize}

\begin{Definition}  [{\cite[Def. 2.8]{CS}}]
Consider as before a well-order $\leq $ on $Q_0 \cup Q_1$ and $\leq_{\omega}$ on $Q_{\geq 0}$.
Let $p \in KQ$ such that $p = \sum_{i=1}^n \lambda_i c_i$ with $\lambda_i \in K^{*}$, $c_i \in Q_{\geq 0}$
and $c_i <_{\omega} c_1$ for all $i\neq 1$. We define  $tip(p):=c_1$. More generally, if $X \subseteq KQ$, we 
define $tip(X) := \{tip(x) \mid x \in X \setminus \{0\}\}$.
Let
\[S := Mintip(I) = \{p \in tip(I) \mid p^\prime \notin tip(I)\mbox{ for all proper divisors $p^\prime$ of $p$}\}.\]
\end{Definition}
From Lemma 2.10 in \cite{CS} it follows that $I$ is
the two sided ideal generated by $\{s- f_s \}_{s\in S}$ for $f_s = \sum_p \lambda_p p$ where the sum is over all $p\in Q_{\geq 2}$ parallel to $s$ and such that $p 
\leq_\omega s$ and $\lambda_p \in K^*$. 
From the reduction system $R$ we obtain a minimal generating set $\cR$  of $I$ by setting 
$\cR = \{ s-f_s \mid (s, f_s) \in R \}$.

We denote by $\mathcal{B}$ the basis
of $A$ whose elements are the images of paths under the canonical map $KQ \to KQ/I$ such that they cannot be reduced using our reduction system.
We will freely refer to elements of $\cB$ as paths.   Let $Q_i$ be the set of paths consisting of $i$ arrows and let $Q_{\geq i} = \cup_{j \geq i} Q_j $. We set $\cB_i = \cB \cap Q_i$. 

Two paths $p$ and $q$ in $Q$ are parallel  if they have same source and target.
If $X$ and $Y$ are sets of paths in $Q$, define 
$$X||Y = \{ (p, q) \in X \times Y | \mbox{$p$ and $q$ are
parallel} \}.$$
The vector spaces $K(X||Y )$ and $\mathrm{Hom}_{E^e} (KX, KY )$ are
isomorphic and we freely  denote by $\alpha || \beta$ an element in
$K(X||Y )$  as well as the morphism in  $\mathrm{Hom}_{E^e}(KX,KY)$
sending  $\alpha $ to $\beta$ and any other basis element to zero.

Next we recall a construction of the Lie bracket on  $\mathrm{HH}^1(A)$ which can be induced from the results  in \cite{CS}.

 Since we are interested in $\mathrm{HH}^1(A)$ and not in the higher Hochschild cohomology spaces, once the basis $\mathcal{B}$ is fixed, 
we can work -for cohomological purposes- with a set of relations $\mathcal{R}$ 
generating the ideal $I$, which does not need to be minimal.

The following is the start of a projective $A$-bimodule resolution of $A$:
\[ \dots \longrightarrow A \otimes_E K\mathcal{R} \otimes_E A \stackrel{{d_1}}{\longrightarrow} A \otimes_E KQ_1 \otimes_E A \stackrel{{d_0}}{\longrightarrow} A \otimes_E A \longrightarrow 0, \]
with differentials
$d_0(1\otimes v \otimes 1)= v \otimes 1 - 1\otimes v$ and, $d_1(1\otimes r \otimes 1)=  
\sum_i\lambda_i\sum_{l=1}^{n_i}v_{j_{1}}\dots v_{j_{l-1}}\otimes v_{j_{l}}\otimes v_{j_{l+1}}\dots v_{j_{n_i}}$
 where $r=\sum_i\lambda_ia_i \in \mathcal{R}$, with
$a_i= v_{j_1}\dots v_{j_{n_i}}$  and $v_{j_k} \in Q_1$.

If we apply  the functor
  $\mathrm{Hom}_{A \otimes A^{op}}(-,A)$ to this chain complex and use the standard natural isomorphism $ \mathrm{Hom}_{A \otimes A^{op}}(A \otimes_E - \otimes_E A,A)  \cong \mathrm{Hom}_{E \otimes E^{op}}( - ,A)$, we obtain  the following 
cochain complex  with the induced differentials:

 \[ 0  \longrightarrow  \mathrm{Hom}_{E\otimes E^{op}}(E, A) \stackrel{\delta^0}{\longrightarrow} \mathrm{Hom}_{E\otimes E^{op}}(KQ_1, A) \stackrel{\delta^1}{\longrightarrow}\mathrm{Hom}_{E\otimes E^{op}}(K\cR, A)  \longrightarrow \dots, \]
 which can also be written as:
 { \[ 0  \longrightarrow  K(Q_0 || \mathcal{B}) \stackrel{\delta^0}{\longrightarrow} K(Q_1 || \mathcal{B}) \stackrel{\delta^1}{\longrightarrow} K(\mathcal {R}|| \mathcal{B}) \longrightarrow \dots \]}
where $\delta^0(e||p) = \sum_{\alpha \in Q_1} \alpha || (\alpha p - p \alpha)  $  and $\delta^1 (\alpha || p) = \sum_{r \in \mathcal{R}, \alpha | r} r || r^{(\alpha, p)}$. In this last expression, if $ r = \sum \lambda_q q$ then $r || r^{(\alpha, p)} = \sum \lambda_q r || q^{(\alpha, p)}$  where 
$r || q^{(\alpha, p)}$   is induced by 
the sum of all the nonzero paths obtained by subsequently replacing each appearance of 
$\alpha$ in $q$ by $p$. If $q$ does not contain the arrow $\alpha$ or if  when replacing
$\alpha$ in $q$ by $p$ we obtain a zero element of $A$ then  we set $q^{(\alpha, p)}= 0$. 

The  Gerstenhaber  structure  defined on the cochains of the above complex and inducing the Gerstenhaber bracket in Hochschild cohomology 
is computed  for example via 
comparison morphisms between a resolution whose first terms are described above and the $E$-reduced 
bar resolution introduced in \cite{Cib}.
  Using the comparison maps between resolutions as in \cite{ALS} - and taking into account that even if in that article the authors only deal with 
Toupie Algebras, the formulas for the comparison morphisms in low degrees do not change -  we have that the
 Lie bracket  on  $\ker(\delta_1) \subset K(Q_1 || \mathcal{B})$ is induced by linearly extending  the following 
\begin{equation}
\label{bracket}
[\alpha || h, \beta || b] = \beta || b^{(\alpha, h)} - \alpha|| h ^{(\beta, b)},
\end{equation}
with $\alpha || h$, $\beta||b$  $\in K(Q_1 || \mathcal{B})$.  This bracket induces a Lie algebra structure on $\mathrm{HH}^1(A)$.


\subsection{Lie algebras of graded algebras}
\label{DefGradedLie}

Let  now $A= KQ/I$ with $I$ an admissible ideal generated by homogeneous  relations, so that the length of paths gives $A$ the structure of a graded algebra 
with arrows in degree one,  and the elements of the basis $\mathcal{B}$ are homogeneous.  This induces a grading on the Lie algebra 
$ \ker (\delta^1)= K(Q_1 || \mathcal{B})\cap \ker \delta^1$ such that the elements in $K(Q_1 || \mathcal{B}_i)\cap \ker \delta^1$ are of degree $i-1$ for all $i \in \mathbb{N}$.
  Moreover, $\mathrm{Im}\delta^0$ is a graded ideal of $ \ker \delta^1$. 
We thus obtain an induced  grading on $\mathrm{HH}^1(A)$.  Set
\begin{equation*}
\begin{split}
\cL_0 &:= K(Q_1|| Q_0) \cap \ker \delta^1\\
\cL_1 &:= K(Q_1 || Q_1) \cap  \ker \delta^1/ \cI_1\\
\cL_i &:= K(Q_1 || \cB_i) \cap  \ker \delta^1/ \cI_i
\end{split}
\end{equation*}
where $$\cI_1=\langle\sum _{a \in Q_1e} a || a - \sum_{a \in eQ_1} a || a \mid e \in Q_0 \rangle$$ and 
for all $i> 1$
 $$\cI_i=\langle\sum _{a \in  eQ_1, \gamma a\in \mathcal{B}} a || \gamma a - \sum_{a \in  Q_1e, 
 a\gamma \in \mathcal{B}} a || a\gamma \mid e || \gamma \in Q_0|| Q_i\rangle$$ 
  
With the above notation  $\cL = \oplus_{i\geq0}\cL_i$
 and $[\cL_i, \cL_j ] \subset \cL_{i+j-1}$ for all
$i, j \geq 0$, where the bracket is induced by the brackets in Equation (\ref{bracket}). Set $\cN:= \oplus_{i \geq 2}\cL_i$. Then  $\cN$ is a 
nilpotent Lie algebra. Moreover, $\mathrm{HH}^1(A) = \cL$.

\section{Criteria for solvability of $\mathrm{HH}^1$ for  finite dimensional algebras}
\label{conditions}

In this section we prove several criteria that if satisfied, imply the solvability of the Lie algebra given by the first Hochschild cohomology of a finite dimensional $K$-algebra.

\subsection{Graded algebras without self-extensions of simples}

\begin{Theorem}
\label{Liestronglysolvable}
Let $K$ be an algebraically closed field  and let $A=KQ/I$ be a  finite
dimensional graded $K$-algebra. Suppose that $\mathrm{Ext}^1_A(S,S)=\{0\}$ for every simple 
$A$-module $S$ and that  $\mathrm{dim}_K(\mathrm{Ext}^1_A(S,T))\leq1$ 
for all nonisomorphic simple $A$-modules $S$ and $ T$. Then $\mathrm{HH}^1 (A)$ is a strongly solvable Lie algebra. 
 In particular, $\mathrm{HH}^1 (A)$ is a solvable Lie algebra.
\end{Theorem}

\begin{proof} 
If $\mathrm{Ext}^1_A(S,S)=\{0\}$ for every simple $A$-module $S$, then $Q$ has no loops. 
Therefore $K(Q_1||\mathcal{B}_0)=\{0\}$. Since $\mathrm{dim}_K(\mathrm{Ext}
^1_A(S,T))\leq1$ 
for every simple $A$-modules $S$ and $T$, there are no parallel arrows in $Q$. 
Equation~(\ref{bracket}) then immediately gives that  $\cL_1$ is an abelian Lie algebra. 
Consequently the derived subalgebra $\cL^{(1)}$ of  $\cL$ is contained in $\cN$. 
Furthermore, $[\cL_i, \cL_j] \subset \cL_{i+j-1}$. As a consequence, $[\cL^{(1)}, \cL^{(1)}] \subset \oplus_{i \geq 3} \cL_i$. By iterating this process, since $A$ is finite dimensional, it follows that $\cL^{(1)}$ is nilpotent and thus $\cL$ is strongly solvable; 
as a consequence $HH^1(A)$ is strongly solvable and hence solvable. 
\end{proof}


\subsection{Graded algebras with self-extensions of simples}
\label{Selfext}
Let us now consider the case of a graded algebra $A$ whose Ext-quiver has loops, that is there is a simple $A$-module $S$ such that 
$\mathrm{Ext}^1_A(S,S) \neq \{0\}$.

For $i \in {\mathbb{N}}$, define the set 
\begin{equation*}
\label{Sigma}
\Sigma_i  (A) := \{ \alpha  || \beta \in Q_1 || \mathcal{B}_i \mid  \alpha || \beta = \alpha_j || \beta_j \mbox{ for some $x =\sum \lambda_j \alpha_j || \beta_j \in  
\ker(\delta^1)$ with $\lambda_j \in K^*$} \}.\end{equation*}

 If it is clear from the context, we will just write $\Sigma_i$ without specifying the algebra.

 One can think of $\Sigma_i  (A)$ as follows: if $\alpha  || \beta \in Q_1 || \mathcal{B}_i$ belongs to this set, this means that there is a 
$K$-linear derivation of $A$ sending $\alpha$ to a linear combination of parallel paths such that $\beta$ appears in this linear combination.

We now state our first result on algebras with self-extensions of simples.

\begin{Theorem}
\label{trivialclasses}
Let $A = KQ/I$ be a finite dimensional graded $K$-algebra such that \sloppy
$\rm{dim_K Ext}^1_A(S,S) \leq1$, for all simple $A$-modules $S$. Suppose further that if $\alpha_i || \alpha_j \in \Sigma_1$, for $i \neq j$, then $\alpha_j || \alpha_i \notin \Sigma_1$  and that  $\Sigma_0$ is empty. Then $\mathrm{HH}^1(A)$ is a solvable Lie algebra. 
\end{Theorem}
\begin{proof}
It immediately follows from the hypotheses that  the Lie algebra $\cL_1^{(1)}$ is abelian.  Since we also have that  $\Sigma_0 = \emptyset$, we obtain 
that  every element in the derived Lie algebra  $\cL^{(2)}$ of $\cL^{(1)}$ has as summands elements in $\Sigma_i$, for $i\geq 2$ and thus $\cL^{(2)} \subset \cN$. Since $\cN$ is nilpotent, it follows that  $\cL^{(2)}$ and hence $\cL^{(1)}$ is nilpotent. Thus $\mathrm{HH}^1(A)$ is solvable.
\end{proof}

\begin{Corollary}
\label{loops}
Let $A=KQ/I$ be a finite dimensional graded $K$-algebra such that $\rm{dim_K Ext}^1_A(S,S) \leq1$, for all simple $A$-modules $S$. Suppose further that if 
$\alpha_i || \alpha_j \in \Sigma_1$, for $i \neq j$, then $\alpha_j || \alpha_i \notin \Sigma_1$ and  that     $\alpha || \alpha^2 $ is not in $\Sigma_2$  
for any loop $\alpha$ in $\cB$.   
 Then $\mathrm{HH}^1(A)$ is a solvable Lie algebra. 
\end{Corollary}
\begin{proof}
Let $\mathcal{L} $ be as above. 
If  $\mathcal{L}_{0} = 0$, the statement follows from Theorem~\ref{trivialclasses}. So assume now that $\cL_{0} \neq 0$. Then  $\alpha || \alpha$ is not in  $\mathcal{L}^{(1)}$ since the only way to obtain it, is as $[\alpha || e, \alpha || \alpha^2]$ for some loop $\alpha$ but by hypothesis such an $\alpha$ does not exist. Then $\mathcal{L}^{(2)}$ does not contain $\alpha || e$ since the only way to obtain it is as $[\alpha || e, \alpha ||\alpha ]$ and $\alpha || \alpha$ is not in $\cL^{(1)}$. Thus the Lie algebra $\mathcal{L}^{(2)}$ satisfies the hypotheses of Theorem \ref{trivialclasses} and by the proof of the same proposition it follows that $\mathcal{L}^{(2)}$ is solvable and therefore $\mathcal {L}$ and hence $\HH^1(A)$ are solvable.
\end{proof}

The following is a consequence of the proof of Corollary~\ref{loops} in characteristic 2.

\begin{Corollary}
\label{loops2}
Let $A=KQ/I$ be a finite dimensional graded $K$-algebra over an algebraically closed field of characteristic $2$ such that 
$\rm{dim_K Ext}^1_A(S,S) \leq1$, for all simple $A$-modules $S$. Suppose further that if $\alpha_i || \alpha_j \in \Sigma_1$, for $i \neq j$, then $\alpha_j || \alpha_i \notin \Sigma_1$.  Then $\mathrm{HH}^1(A)$  is a solvable Lie algebra. 
\end{Corollary}

\begin{proof}
Since the characteristic of $K$ is $2$, we have for any loop $\alpha$ at vertex $e$ that  
$[\alpha || e, \alpha || \alpha^2]=0$. Therefore we can apply the same arguments as in Corollary~\ref{loops} and the result follows. 
\end{proof}

\subsection{Graded Local algebras}

We finish this section by considering the case when $A$ is a graded local algebra.

\begin{Proposition} \label{localB0}
Let $A=KQ/I$ be a local finite dimensional graded $K$-algebra such that $\Sigma_0 = \emptyset$. 
Suppose $Q$ has  loops   $\alpha_1, \alpha_2, \ldots, \alpha_n$. 
If  none of the $\alpha_i || \alpha_j$, are in  $\Sigma_1$ for $i \neq j$ and $1 \leq i, j \leq n$, then $\mathrm{HH}^1(A)$ is solvable.
\end{Proposition}
\begin{proof}
Let $\cL$ be as above with $\cL_0 = 0$. 
Because of the hypotheses on $\Sigma_1$, we have that   $\mathcal{L}_1$ is abelian 
and consequently $\mathcal{L}_1$ is solvable. Furthermore, $\cL = \cL_1 \oplus \cN$ and since $[\cL_1, \cN] \subset \cN$ and $\cN$ is nilpotent,  the statement follows. 
\end{proof}

The next result is for the more general case that $\Sigma_0$ is not empty. 

\begin{Theorem}\label{thm:graded loops}
\label{local}
Let $A=KQ/I$ be a local finite dimensional graded $K$-algebra with two loops $\alpha_1$ and $\alpha_2$ such that $
\alpha_1 || \alpha_2$  and $\alpha_2 || \alpha_1$ are not in  $\Sigma_1$.  
\begin{enumerate}
\item If ${\rm char}(K) = 2$ then $\mathrm{HH}^1(A)$ is solvable.
\item If ${\rm char}(K) \neq 2$ and $\alpha_i || \alpha^2_i \notin \Sigma_2$ for $i\in \{1,2\}$ then $\mathrm{HH}^1(A)$ is solvable.
\end{enumerate}
\end{Theorem}

\begin{proof}
Let $\mathcal{L}$ be as above. If $\Sigma_0$ is empty then the statement follows from  Proposition \ref{localB0}. So assume that $\Sigma_0$ is not empty. 

First we consider the case when both $\alpha_1|| e$ and $\alpha_2 || e$ are in $\Sigma_0$. Then the Lie algebra 
$\mathcal{L}$ does not contain $\alpha_i||\alpha_i\alpha_j$ and $\alpha_i||\alpha_j\alpha_i$ for $i, j\in \{1,2\}$, $i\neq j$, 
because otherwise the derived subalgebra $\cL^{(1)}$ contains $[\alpha_i || e, \alpha_i || \alpha_i\alpha_j]=  \alpha_i||\alpha_j$ and 
$[\alpha_j || e, \alpha_j|| \alpha_i\alpha_j]= \alpha_j||\alpha_i$ which are not in $\mathcal{L}$ since they are not in $\Sigma_1$ 
by hypothesis. 
Using the same argument as in Corollary~\ref{loops},  $\mathcal {L}^{(1)}$  does not contain $\alpha_i || \alpha_i$ for $i=\{1,2\}$ since the only way to obtain them is through the bracket $[\alpha_i ||\alpha^2_i, \alpha_i || e] = - 2 \alpha_i || \alpha_i$ for $i \in \{1,2\}$ which in characteristic 2 is equal to zero and if the characteristic is different from 2, by hypothesis, the elements $\alpha_i || \alpha_i^2$ are not in $\cL$ since they are not in $\Sigma_2$. The Lie algebra $\mathcal{L}^{(2)}$ does not contain $\alpha_i || e$ for $i=\{1,2\}$ since the only way to obtain them is through the bracket $[\alpha_i ||\alpha_i, \alpha_i || e]$ for $i \in \{1,2\}$ but $\alpha_i ||\alpha_i$ are not in $\mathcal{L}^{(1)}$. Therefore by Corollary~\ref{loops} the Lie algebra $\mathcal{L}^{(2)}$ is solvable and consequently $\mathcal{L}$ is solvable.

In case that only one of $\alpha_1|| e$ and $\alpha_2 || e$ are in $\Sigma_0$, the argument is similar.  \end{proof}


\subsection{The general case of not necessarily graded algebras}

In this subsection we will see how to deal with not necessarily graded finite dimensional 
algebras. For this, we will make use of results by Eisele and Raedschelders \cite{ER}.

We start by recalling some definitions and basic facts.

\begin{Definition}
Given a  $K$-algebra $A$, define
\[ {\rm Der}_{rad}(A):= \{\delta \in {\rm Der}_K(A)|\hbox{  } \delta(J(A)) \subseteq J(A) \},  \]
which is a Lie subalgebra of ${\rm Der}_K(A)$. 
\end{Definition}

\begin{Remark} The inner derivations of $A$ form a Lie ideal of ${\rm Der}_{rad}(A)$ -see \cite{ER}- so, there is an inclusion:
\[  \mathrm{HH}^1_{rad}(A):= {\rm Der}_{rad}(A)/Inn(A) \hookrightarrow  {\rm Der}_K(A)/Inn(A)= \mathrm{HH}^1(A) \label{embedding}.  \]
As a consequence, $\mathrm{HH}^1_{rad}(A)$ is a Lie subalgebra of $\mathrm{HH}^1(A)$.
\end{Remark}

The problem then is: how far is $\mathrm{HH}^1_{rad}(A)$ from $\mathrm{HH}^1(A)$? 
In other words, how far is the inclusion in Remark~\ref{embedding} from being an isomorphism? In fact, in characteristic zero, this is always 
an isomorphism, as proved by Hochschild in 
\cite[Thm. 4.2]{Hoch}. In Hochschild's terminology, this means that $J(A)$ is a characteristic ideal of $A$.

This is no longer true in characteristic $p>0$, as the following example shows: consider $A=K[x]/(x^p)$ and the $K$-linear derivation  $\delta_0: A\to A$ such that $\delta_0(x)=1$. It is well defined, and clearly it does not preserve the radical.
In this example any derivation is a scalar multiple  of $\delta_0$ plus a derivation preserving the radical, so the quotient $\mathrm{HH}^1(A)/\mathrm{HH}^1_{rad}(A)$ is $1$-dimensional.

In \cite{ER} the authors provide a sufficient condition for $\mathrm{HH}^1_{rad}(A)$ to be equal to $\mathrm{HH}^1(A)$. Their result reads as follows.

\begin{Proposition} \cite[Prop. 2.7]{ER}
Assume that $A = KQ/I$. For every loop $\alpha \in Q_1$, let $n_\alpha$ denote the minimal integer such that $ \alpha^{n_\alpha} \in J(A)^{n_\alpha +1}$. If $p$ does not divide $\prod_\alpha n_\alpha$, then
\[  {\rm Der}_{rad}(A)= {\rm Der}(A) \hbox{ and } \mathrm{HH}^1_{rad}(A)= \mathrm{HH}^1(A).  \]
\end{Proposition}

However, they also showed in \cite[Rem. 2.8]{ER} that the previous condition is not necessary. As one can deduce at this point, the difference between 
$\mathrm{HH}^1_{rad}(A)$ and $\mathrm{HH}^1(A)$
comes from loops in $Q$. Given the algebra $A=KQ/I$, with $\rm{char}(K)=p$, such that $Q$ contains loops, even if the hypotheses of \cite[Prop. 2.7]{ER}, are not satisfied, it is in fact possible 
to verify whether $\mathrm{HH}^1_{rad}(A)$ equals $\mathrm{HH}^1(A)$ or not. 

Before stating the next lemma, we adopt the following notation. We  say that a finite linear combination of paths in $KQ$ is in its {\em reduced form} if it is written as 
$\sum_i \lambda_ip_i$ where all the coefficients $\lambda_i \in K$ are not zero and all the paths $p_i \in Q$ are different.

\begin{Lemma}\label{rad=nonrad} Let $KQ/I$ be a finite dimensional algebra and let 
 $R$ be a minimal set of relations generating the ideal $I$  such that for every loop $\alpha$ in $Q$ there is a relation $r$ in 
$R$ which  has as a 
summand a path containing $\alpha$.  Consider the derivation $d: KQ\to KQ$ such that $d(\alpha)=s(\alpha)$ and such that $d$  is the identity on any other arrow. Suppose that at least one of the paths that appear in the reduced form of $d(r)$ obtained by replacing one occurrence of $\alpha$ by $s(\alpha)$ does not appear in a relation in the algebra $A$. Suppose further that if the characteristic of $K$ is p, the relation $r$ is not of the form $\alpha^{ps}$. 
Then  $\mathrm{HH}^1_{rad}(A)$ and $\mathrm{HH}^1(A)$ coincide. 
\end{Lemma}

As a particular case of  Lemma~\ref{rad=nonrad} we obtain Remark 2.8 of \cite{ER}.
\begin{proof}
 Given a loop $\alpha$ in $Q$, let $r$ be a relation as in the statement. Let $c_1\dots c_n$ be a summand of $r$ and let $c_j=\alpha$ and write 
$r=c_1\dots c_n + \tilde{r}$. Let $d:A\to A$ be a $K$-linear derivation.
In $A$, we have that
\[ 0= d(r)=  d(c_1\dots c_n) + d(\tilde{r})= c_1\dots c_{j-1}c_{j+1} \dots c_n + d(\tilde{r}) \]
and that $c_1\dots c_{j-1}c_{j+1} \dots c_n$ does not cancel with any term in $d(\tilde{r})$.
But this is impossible, since by hypothesis $c_1\dots c_{j-1}c_{j+1} \dots c_n$ is not a summand in a relation in $A$. 

In characteristic $p$, we have to omit relations containing $\alpha^{ps}$ to avoid the corresponding sum of terms in $d(r)$  being zero because of the characteristic.
\end{proof}

Going back to the problem of solvability, the reduction of the general case to the radical square zero case makes use of Proposition 2.4 and 
Corollary 2.5 of \cite{ER}. For the convenience of the reader we include the statements here.

\begin{Proposition}\cite[Prop. 2.4]{ER}
Let $A$ be a finite dimensional $K$-algebra, and $3 \le  n \in \mathbb{N}$. There is a morphism of Lie algebras
\[ {\rm Der}_{rad}(A/J(A)^n) \to {\rm Der}_{rad}(A/J(A)^{n-1}) \]
with abelian kernel. If ${\rm Der}_{rad}(A/J(A)^{n-1})$ is solvable, then so is ${\rm Der}_{rad}(A/J(A)^{n})$. 
\end{Proposition}

\begin{Corollary}\cite[Cor. 2.5]{ER}
Let $A$ be a finite dimensional $K$-algebra, and $3 \le  n \in \mathbb{N}$. There is a morphism of Lie algebras
\[ \mathrm{HH}^1_{rad}(A/J(A)^n) \to \mathrm{HH}^1_{rad}(A/J(A)^{n-1}) \]
with abelian kernel. If $\mathrm{HH}^1_{rad}(A/J(A)^{n-1})$ is solvable, then so is $\mathrm{HH}^1_{rad}(A/J(A)^n)$. In particular, 
if $\mathrm{HH}^1_{rad}(A/J(A)^2)$
is solvable, then so is $\mathrm{HH}^1_{rad}(A)$.
\end{Corollary}

We now prove analogues in the general case of the results which we have already proved for the graded case. 

The first thing to remark is that, as we will show,  Theorem \ref{Liestronglysolvable} is still true in the ungraded case.

\begin{Theorem}
\label{Liestronglysolvable-general}
Let $K$ be an algebraically closed field and let $A$ be a  finite
dimensional 
$K$-algebra. Suppose that $\mathrm{Ext}^1_A(S,S)=\{0\}$ for every simple 
$A$-module $S$ and that  $\mathrm{dim}_K(\mathrm{Ext}^1_A(S,T))\leq1$ 
for all nonisomorphic simple $A$-modules $S$ and $ T$. Then $\mathrm{HH}^1 (A)$ is a strongly solvable Lie algebra. 
In particular, $\mathrm{HH}^1 (A)$ is a solvable Lie algebra.
\end{Theorem}
\begin{proof}
Given $A$, since  the Ext-quiver has no loops, we know that $\mathrm{HH}^1(A)$ and $\mathrm{HH}^1_{rad}(A)$ are equal, 
and if  $\mathrm{HH}^1_{rad}(A/J(A)^2)$ is solvable, then $\mathrm{HH}^1_{rad}(A)$ is solvable too.  Notice that 
\[   {\rm Ext}^1_{A/J(A)^2}(S,S)= {\rm Ext}^1_A(S,S)= \{0\} \hbox{ and }  \]
\[ {\rm dim_KExt}^1_{A/(J(A))^2}(S,T)= {\rm dim_KExt}^1_A(S,T)\le 1  \]
for all simple $A$-modules $S,T$.
So, $A$ satisfies the hypotheses if and only if $A/J(A)^2$, which is is a graded algebra, does for Theorem \ref{Liestronglysolvable}.
As a consequence, $\mathrm{HH}^1(A/J(A)^2)$ is solvable and so is $\mathrm{HH}^1(A)$.
\end{proof}

We now look at the behaviour of the $\Sigma_j$'s when one passes from $A$ to $A/J(A)^i$.

\begin{Remark}
For $j \ge i$, it is clear that $\Sigma_j(A/J(A)^i)=0$; moreover $\Sigma_{i-1}(A/J(A)^i)=K(Q||\mathcal{B}_i)$ for all $i\ge 2$. 
\end{Remark}

Next we describe $\Sigma_0(A/J(A)^i)$, recall that:
\[ \Sigma_0(A/J(A)^i)= \{ \alpha || e | \hbox{  } \alpha || e \hbox{ is a summand of an element in } \ker(\delta^1) \} \]
and  that $\delta^ 1(\alpha || e)=  \sum_{r \in \mathcal{R}, \alpha | r} r || r^{(\alpha, e)}$.
Since we are supposing that $A$ is finite dimensional and $\alpha$ is a loop, at least one relation containing $\alpha$ must exist, so that the 
indexing set of the sum is non empty. For $i=2$, 
\[ \delta^ 1(\alpha || e)=  \sum_{\gamma \in Q_1|\gamma \neq \alpha, t(\gamma)=e } r_{\alpha\gamma} || \gamma + 
\sum_{\mu \in Q_1|\mu \neq \alpha, s(\mu)=e } r_{\mu\alpha} || \mu + \alpha^2|| 2\alpha,\]
where $r_{\alpha\gamma}$ and $r_{\mu\alpha}$ are the relations corresponding to the fact that the respective compositions of these arrows are zero.

If the characteristic is not $2$, then the last term is non zero  and implies that $\alpha || e$ will never appear as a summand of an element in $\ker(\delta^ 1)$.
In characteristic $2$, the same conclusion can be obtained by looking either at the first sum or at the second one, since at least 
one of them is indexed over a non empty set, 
unless the algebra $A$ is $K[x]/(x^n)$, for some $n\ge 2$.

As a consequence, the only case where $\Sigma_0(A)$ is empty and $\Sigma_0(A/J(A)^2))$ is non empty is when 
$A= K[x]/(x^n)$, for some $n> 2$, $char(K) = 2$ and $n$ is odd.

In the following we show the ungraded analogues of  Corollary~\ref{loops}, Corollary\ref{loops2} and  Theorem~\ref{trivialclasses} in a more restricted form. We note that each time that we will use them in the sequel, we will first verify either by hand or by applying Lemma~\ref{rad=nonrad} that $\mathrm{HH}_{rad}^1(A)= \mathrm{HH}^1(A)$.

\begin{Theorem}
\label{trivialclasses-nongraded}
Let $A = KQ/I$ be a finite dimensional $K$-algebra  such that \sloppy
$\rm{dim_K Ext}^1_A(S,T) \leq1$ for  simple $A$-modules $S$ and $T$. Suppose that $\Sigma_0(A)$ is empty. Then $\mathrm{HH}^1(A)$ is a solvable Lie algebra. 
\end{Theorem}
\begin{proof}
As before, we consider the graded algebra $A/J(A)^2$. We note that since there are no parallel arrows in $Q$, the hypotheses of 
Theorem~\ref{trivialclasses} 
are satisfied by $A/J(A)^2$, hence $\mathrm{HH}^1(A/J(A)^2)$ is solvable and so the same holds for $\mathrm{HH}_{rad}^1(A)$. By Lemma~\ref{rad=nonrad}, it follows that $\HH^1(A)$ is solvable. The remaining case is when $A = K[x]/(x^n)$ for some $n > 2$, $char(K) = 2$ and $n$ is odd. 
In this case $\Sigma_0(A)$ is empty but $\Sigma_0(A/J(A)^2)$ is not, therefore we cannot apply the same
method as before. However, $\HH^1(A)$ has as a $K$-basis the derivations $f_i$ such that $f_i(x) = x^i$ for $1\le i \le n-1$.
It is easy to show that $\HH^1(A)$ is solvable.
\end{proof}

\begin{Proposition}
\label{loops-nongraded}
Let $A=KQ/I$ be a finite dimensional $K$-algebra 
such that $\rm{dim_K Ext}^1_A(S,T) \leq1$ for all simple 
$A$-modules $S$ and $T$. Suppose that $\alpha || \alpha^2 $ is not in $\Sigma_2(A)$  for any loop $\alpha$ in $Q_1$.   
 Then $\mathrm{HH}_{rad}^1(A)$ is a solvable Lie algebra. 
\end{Proposition}

\begin{proof}
Again, we consider the graded algebra $A/J(A)^2$ and we note that since there are no parallel arrows in $Q$, it satisfies the hypotheses of Corollary~\ref{loops}.
The result follows.
\end{proof}

Finally we  state the analogue of Corollary~\ref{loops2}. After reducing to $A/J(A)^2$, its proof is 
analogous to the proof of Corollary~\ref{loops2}.

\begin{Corollary}
\label{loops2-nongraded}
Let $A=KQ/I$ be a finite dimensional algebra over a field $K$ of characteristic $2$ 
such that 
$\rm{dim_K Ext}^1_A(S,T) \leq1$ for all simple $A$-modules $S$ and $T$. Then $\mathrm{HH_{rad}}^1(A)$ is a solvable Lie algebra.
\end{Corollary}

\section{Applications}
\label{applications}
In this section we prove the solvability of $\mathrm{HH}^1(A)$ for the symmetric tame algebras $A$ classified in \cite{Skow}, as an application of the results  in Section~\ref{conditions}. As before, throughout  this section $K$ is an algebraically closed field. 


\subsection{Dihedral, semi-dihedral and quaternion algebras}

We first focus on symmetric tame algebras of dihedral, semi-dihedral and quaternion type, studied and classified up to Morita equivalence and up to scalars in \cite{Erd}. This list of algebras contains all tame blocks of group algebras of finite groups. The  classification in \cite{Erd},  has been extended up to derived equivalences in \cite{Holm} and 
more recently most of the algebras  of dihedral, semi-dihedral and quaternion  have been distinguished  up to stable equivalence of Morita type~\cite{ZZ}. For algebras of dihedral type this  classification is complete \cite{Tai}.

\begin{Theorem}
\label{DSQ}
Let $K$ be an algebraically closed field of arbitrary characteristic. Let $A$ be a symmetric tame algebra of dihedral, semi-dihedral or quaternion type not derived equivalent to $K[X,Y]/(X^2, Y^2)$ if the characteristic of $K$ is not 2. Then $\mathrm{HH^1}(A)$ is a solvable Lie algebra.
\end{Theorem}

\begin{proof}
We organise the proof by the number of simple modules, and within each case we consider the dihedral, semi-dihedral and quaternion cases up to derived equivalence.

In the local algebra case all algebras have the same  quiver $Q$ which has one vertex $e$ and  two loops $X$ and $Y$.

By \cite{ Holm} in dihedral type we have that  up to derived equivalence $A$ is one of the following:

\begin{enumerate}
\item $D(1\mathcal{A})^1_1= KQ/(X^2,Y^2, XY-YX)$, 
\item $KQ/ (XY, YX, X^m-Y^n)$: for $m\geq n\geq 2, m+n>4$, 
\item $D(1\mathcal{A})^k_1 = KQ/ (X^2,  Y^2, (XY)^k- (YX)^k)$ in $k\geq 2$; 

\noindent \hspace{-1.4cm} and when $\mathrm{char}(K)=2$, there are two more cases: 

\item $KQ/(X^2, XY-YX, XY-Y^2)$, 
\item $D(1\mathcal{A})^k_2(d) = KQ/(X^2-(XY)^k,Y^2-d\cdot(XY)^k,(XY)^k-(YX)^k,(XY)^kX,(YX)^kY)$, for $ k\geq 2, d \in \{0,1\}$.
\end{enumerate}

In case (1), it is well-known, see for example \cite{Jacob},  that
$\mathrm{HH}^1(A)$ is a Jacobson--Witt algebra which is a simple Lie algebra of Cartan type, so, in particular, it is not solvable On the other hand,
for any other characteristic $\HH^1$ maybe solvable. 

For case (2), we note that  the algebra $A /J(A)^3$ is graded for all  $m \geq n \geq 2$ and  $m+n > 4$.  
There are two different cases, (a) is for $m,n \ge 3$ and (b) is for  $m\ge 3 $ and $n=2$. In both cases one easily checks that $\Sigma_0(A)$ is empty 
In case (a), it is also easy to verify that $X||Y$ and $Y||X$ are not $\Sigma_1(A)$. Thus by Proposition~\ref{localB0} we have that $\HH^1(A/J(A)^3)$ is solvable. 
Using Lemma \ref{rad=nonrad} and \cite[Cor. 2.5]{ER} we obtain the result in this case.

Now suppose for case (b) that  $n=2 $ and  $m\ge 3$.  Then an easy calculation shows that for $A /J(A)^3$, we have 
$\ker \delta^1 = \langle X||X, X||Y, X||X^2, Y||Y,  Y||X^2\rangle$. One easily verifies by repeatedly calculating  the brackets of the elements in 
$\ker \delta^1$, that after a small number of iterations, they are all eventually zero.  Therefore $\HH^1(A/J(A)^3)$ and hence $\HH^1(A)$  which
coincides with $\HH^1_{rad}(A)$ are solvable.

In case (3), the algebra $A$ is graded. A direct calculation shows that $X||Y$ and $Y||X$ are not $\Sigma_1$ and that if 
$\mathrm{char}(K)\neq 2$, then $X||e$ and $Y||e$ are not $\Sigma_0$. The result then follows from Theorem~\ref{local} (1) if  $\mathrm{char}(K)= 2$ and from  Proposition~\ref{localB0} if $\mathrm{char}(K)\neq 2$. 

In case (4) since $A$ is graded, it is enough to verify that $X||e$ and $Y||e$ are not in $\Sigma_0$ and that the Lie algebra $\cL_1$  has $K$-basis  
$\{  X||X + Y||Y, X||X+ X||Y,Y||X \}$.   
 Moreover, $\cL_1^{(1)}$ is abelian and thus  $\mathrm{HH}^1(A)$ is solvable.

For case (5)  we note that the algebra $A/ J(A)^3$ is isomorphic to  $KQ/(X^2, Y^2, XYX, YXY)$, which is graded. Since
$X||e$ and $Y||e$ are not in $\Sigma_0$, and $X||Y$ and  $Y||X$ are not in $\Sigma_1$, the result follows by Proposition~\ref{localB0} and Lemma~\ref{rad=nonrad}.

Next we consider local algebras of semi-dihedral type. In this situation, there are two possibilities, one of which occurs only in characteristic $2$.

(1) $SD(1\mathcal{A})^k_1 = KQ/( (XY)^k-(YX)^k,(XY)^kX,Y^2,X^2-(YX)^{k-1}Y)$, for $k\geq 2$.  We consider $A/J(A)^3$ which is again given bt 
$KQ/(X^2, Y^2, XYX, YXY)$.
It is easy to verify that $X||Y$ and $Y||X$ 
are not in  $\Sigma_1(A/J(A)^3)$. 
In addition, if $\mathrm{char}(K)\neq 2$, $X||e$ and $Y||e$ are not in $\Sigma_0$.  
Hence the solvability of $\mathrm{HH}^1(A)$ follows from Proposition \ref{localB0} and Lemma~\ref{rad=nonrad}.

(2) $SD(1\mathcal{A})^k_2(c, d) = KQ/ ( (XY)^k-(YX)^k,(XY)^kX,Y^2-d(XY)^k, X^2-(YX)^{k-1}Y+c(XY)^k)$, $\mathrm{char}(K)=2$, $k\geq 2$, $(c,d)\neq (0,0)$. In this case, we also  focus  on the algebra $A/J(A)^3$, which is again given by  $KQ/(X^2, Y^2, XYX, YXY) $ and we proceed as before.

Finally, we consider local algebras of quaternion type. Again, there are two possibilities with  one of them occurring only in characteristic $2$. 

(1) $Q(1A)^k_1 = KQ/  ((XY)^k-(YX)^k,(XY)^kX,Y^2-(XY)^{k-1}X,X^2-(YX)^{k-1}Y)$, for $k\geq 2$

and

(2)  \sloppy the algebras $Q(1\mathcal{A})^k_2(c, d)$, which occur only in characteristic 2 and are given by $$KQ/ (X^2-(YX)^{k-1}Y-c(XY)^k,Y^2-(XY)^{k-1}X-d(XY)^k,(XY)^k-(YX)^k,(XY)^kX,(YX)^kY),$$ for $k\geq 2$ and $(c,d)\neq (0,0)$.

In both cases $A/J(A)^3= KQ/(X^2, Y^2, XYX, YXY) $ and the same arguments apply.

Up to derived equivalence, there are four families of symmetric tame algebras with two simple modules and they 
all have the same quiver Q:

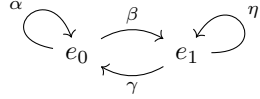
\[
\begin{tikzcd}
e_0 \arrow[out=170,in=110,loop,"\alpha"]
\arrow[r,bend left,"\beta"] 
&
e_1 \arrow[out=0,in=60,loop,swap,"\eta"]
 \arrow[l, bend left,"\gamma"] 
\end{tikzcd}
\]

In characteristic $2$ the solvability of $\mathrm{HH}^1(A)$ follows from Corollary 
\ref{loops2}, Corollary \ref{loops2-nongraded} and Lemma \ref{rad=nonrad}, so it is enough to consider the case $\mathrm{char}(K)\neq 2$.

\begin{enumerate}
\item $D(2\mathcal{B})^{k,s}(c) = KQ/ (\beta\eta, \eta\gamma, \gamma\beta, \alpha^2- c(\alpha\beta\gamma)^k, (\alpha\beta\gamma)^k - (\beta\gamma\alpha)^k, \eta^{s} - (\gamma\alpha\beta)^k)$ with $k\ge s\ge 1$ and $c\in \{0,1\}.$ 
By length reasons, $\eta || e_1$ and $\alpha || e_0$ are not in $\Sigma_0$.

\item $SD(2\mathcal{B})^{k,t}_1 (c) = KQ/ ( \gamma\beta, \eta\gamma, \beta\eta, \alpha^2 - (\beta\gamma\alpha)^{k-1}\beta\gamma- c(\alpha\beta\gamma)^k , \eta^{t}-  (\gamma\alpha\beta)^k, (\alpha\beta\gamma)^k - (\beta\gamma\alpha)^k) $  with $k\ge 1, t\ge 2$ and $c\in \{0,1\}$.
Analogously to case $D(2B)^{k,s}(c)$, we deduce that $\eta || e_1$ is 
not in $\Sigma_0$, and neither is $\alpha || e_0$.

\item $SD(2\mathcal{B})^{k,t}_2 (c) = KQ /(\beta\eta - (\alpha\beta\gamma)^{k-1}\alpha\beta, \eta\gamma -(\gamma\alpha\beta)^{k-1}\gamma\alpha, \gamma\beta -\eta^{t-1}, \alpha^2-  c(\alpha\beta\gamma)^k, \beta\eta^2, \eta^2\gamma) $  with $k\ge 1, t\ge 2$, $k+t\ge 4$ and $c\in \{0,1\}$.
The elements $\eta || e_0$ and $\alpha || e_0$ are not in in $\Sigma_0$.

\item $Q(2\mathcal{B})^{k,s}_1 (a, c)$: the algebras are of the form $$ KQ/( \gamma\beta- \eta^{s-1},\beta\eta- (\alpha\beta\gamma)^{k-1}\alpha\beta, \eta\gamma-(\gamma\alpha\beta)^{k-1}\gamma\alpha, \alpha^2- a(\beta\gamma\alpha)^{k-1}\beta\gamma-c(\beta\gamma\alpha)^{k}, \alpha^2\beta, \gamma\alpha^2).
$$ with $k\ge 1, s\ge 3$ and $a\neq 0$.
 Again, by length reasons, $\eta || e_1$ and $\alpha || e_0$ are not in $\Sigma_0$.
\end{enumerate}

For these algebras, we can use Theorem \ref{trivialclasses-nongraded}, since in all cases $\eta || e_1$ and $\alpha || e_0$ are not in $\Sigma_0$
by length reasons and so $\Sigma_0$ is empty.

For tame symmetric algebras with $3$ simple modules, there are three classes of algebras denoted by $3\mathcal{K}$, $3\mathcal{A}$ and $D(3\mathcal{R})^{k,s,t,u}$ with $s\geq t\geq u\geq k \geq 1$, $t\geq 2$. The quiver of the algebras in $3\mathcal{K}$ is of the form 
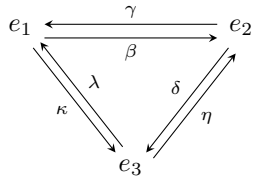
\[
\begin{tikzcd}[arrow style=tikz,>=stealth,row sep=4em]
e_1  
\arrow[rr,shift right=1.2ex, swap, "\beta"]
  \arrow[dr,shift right=.4ex,swap,"\kappa"]
&& e_2 \arrow[ll, swap,"\gamma"]
  \arrow[dl,shift left=.4ex, swap,"\delta" ]
\\
& e_3 
  \arrow[ur,shift right=1.2ex,swap,"\eta"]
      \arrow[ul,shift right=.4ex,swap,"\lambda"]
\end{tikzcd}
\]

while for algebras of type $3\mathcal{A}$  the quiver is:
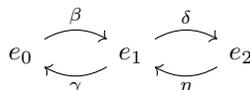
\[
\begin{tikzcd}
e_0
\arrow[r,bend left,"\beta"] 
&
e_1
\arrow[r,bend left,"\delta"] 
 \arrow[l, bend left,"\gamma"] 
 &
 e_2
 \arrow[l, bend left,"\eta"] 
\end{tikzcd}
\]

 Thus the solvability of $\mathrm{HH}^1$ for algebras of 
type $3\mathcal{K}$ and of type $3\mathcal{A}$  follows from Theorem \ref{Liestronglysolvable-general}. 

The algebras of type  $D(3\mathcal{R})^{k,s,t,u}$ with $s\geq t\geq u\geq k \geq 1$, $t\geq 2$ have the following quiver: 

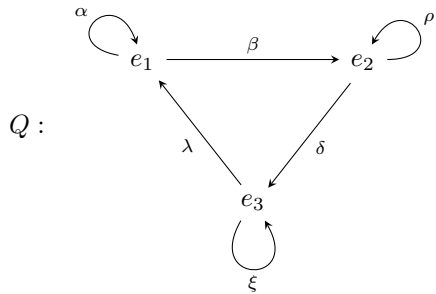
\[Q: 
\begin{tikzcd}[arrow style=tikz,>=stealth,row sep=4em]
e_1   \arrow[out=170,in=110,loop,"\alpha"]
\arrow[rr,  "\beta"]
&& e_2 \arrow[out=0,in=60,loop,swap,"\rho"]
  \arrow[dl,shift left=.4ex, "\delta" ]
\\
& e_3 \arrow[out=240,in=300,loop,swap,"\xi"]
      \arrow[ul,"\lambda"]
\end{tikzcd}
\]
and can be presented as follows:  
$KQ/ (\alpha\beta, \beta\rho, \rho\delta, \delta\xi, \xi\lambda, \lambda\alpha, \alpha^{s}-(\beta\delta\lambda)^k, \rho^t-(\delta\gamma\beta)^k, \xi^u-(\lambda\beta\delta)^k).$
 
The result holds in case $\mathrm{char}(K)=2$ by  Corollary  \ref{loops2-nongraded} and in any other characteristic, it follows from 
Theorem \ref{trivialclasses-nongraded}, after verifying that $\Sigma_0(A)=0$ due to the monomial quadratic relations.
\end{proof}


\subsection {Brauer graph algebras}
In this  subsection we prove the solvability of $\mathrm{HH}^1(A)$ for $A$ a Brauer graph algebra not derived equivalent to the trivial  extension of the Kronecker algebra. In the latter case, $\mathrm{HH}^1(A)$ is isomorphic to $\mathfrak{gl}_2(K)$  which is not solvable except in characteristic 2.   

The solvability of the first Hochschild cohomology of Brauer graph algebras with multiplicity function identically equal to one and 
not derived equivalent to the trivial extension of the Kronecker algebra has been shown in \cite{CSS}. We now prove that the first Hochschild 
cohomology of any Brauer graph algebra with any multiplicity function (apart from the trivial extension of the Kronecker algebra) 
is solvable in any characteristic, except possibly in characteristic 2. We call a relation of the form $p-q$, for paths $p, q$ in $Q$, a \emph{binomial relation}.

\begin{Theorem}
\label{Brauer}
Let $A = KQ/I$ be a Brauer graph algebra. Then $\HH^1(A)$ is a solvable Lie algebra except if $A$ is derived equivalent to the trivial extension of the Kronecker algebra in characteristic different from 2 or if the characteristic of $K$ is 2 and there is a loop $\alpha$ such that $\alpha^2 =0$. 
\end{Theorem}

As we will see in the proof, for a Brauer graph algebra, we have  in almost all 
cases that $\HH^1_{rad}(A) = \HH^1(A)$ and only if the characteristic of $K$ is 2, we might not 
have equality. However,  a direct calculation should show  that  $\HH^1(A)$ is solvable for all cases 
also in characteristic 2 as is shown for some of these cases in Section~\ref{sec:weakly symmetric} and for Brauer graph algebras with multiplicity function identically equal to one in \cite{CSS}.

 \begin{proof}[Proof of Theorem~\ref{Brauer}]
  Let $A = KQ/I$ be a Brauer graph algebra not derived equivalent to the trivial extension of the Kronecker algebra. Brauer graph algebras with one simple module fall into the local dihedral algebras of type (1)-(3) in the proof of Theorem~\ref{DSQ} and the result then follows in this case. So we now assume that $A$ has at least two non-isomorphic simple $A$-modules.  Then $B:= A/J(A)^3$ is a graded algebra with quiver $Q$ which might have loops and parallel arrows. Since $A$ is special biserial non-local, there is at most one loop at any given vertex.  Denote by $\cR_{mon}$ the set of monomial relations of $B$. 
 
Suppose that $\alpha$ and $\beta$ are two parallel arrows in $Q_1$. Then there exists no monomial relation $r \in \cR_{mon}$  such that  both 
$\alpha $ and $\beta$ are in $r$. But there exists at least one monomial relation $r_0 \in \cR_{mon}$ 
containing at least one of them, say, for example, $r_0$ contains $\alpha$. Then $r_0 || r_0^{(\alpha, 
\beta) }\neq 0$. Suppose that $\gamma$ is another arrow in $r_0$. Then $\alpha$ is not parallel to 
$\gamma$ and $r_0 || r_0^{(\alpha, \beta)} \neq r_0 || r_0^{(\gamma, \delta)}$, for any $\delta \in 
Q_1$ with $\delta || \gamma$. Thus  $\sum_{\alpha \in r | r \in \cR_{mon}} r || r^{(\alpha, \beta)} 
\neq 0$ and $\alpha || \beta \notin \ker \delta^1$.

Hence, by Corollary~\ref{loops2} if the characteristic of $K$ is 2, we have  that $\HH^1(B)$ is solvable and thus $\HH^1_{rad}(A)$ is solvable.

 If the characteristic of $K$ is not 2, we show that the hypotheses of Theorem~\ref{trivialclasses} hold for $B$.  By the above it is  enough to show 
that $\Sigma_0$ is empty. Let $\alpha \in Q_1$ be a loop at vertex $e$. Then $\alpha$ must appear in at least one relation $r$  generating $I$. 
This relation is either 
of the form $r = \alpha^2 $ or of the form $ r = \alpha^m - C^n $ for some integers $n,m \geq 1$ and a cycle $C = c_1 \ldots c_k$ in $Q$.

If $r = \alpha^2$ then $r$ is also a relation for $B$ and $\delta^1(\alpha ||e)$ contains as non-zero summand $2 r || \alpha$.

If $r = \alpha^m - C^n$ then $\alpha c_1$ and $c_k \alpha$ are  monomial relations for $B$ and $\delta^1(\alpha ||e) $ contains as non-zero summand $\alpha c_1 || c_1$. 

Thus the hypotheses of Theorem~\ref{trivialclasses} are verified and $\HH^1(B)$ is solvable. It then follows from  \cite[Cor. 2.5]{ER} that $\HH^1_{rad}(A)$ is solvable.

 Now suppose that every loop $\alpha$ in $A$ is in a binomial relation of the form $\alpha^m - C^n$ for some nonzero cycle $C= c_1 \ldots c_k $ and $m \geq 2$ and $n \geq 1$. Then both $\alpha c_1 $ and $c_k \alpha$ as well as $\alpha^{m+1}$ are in $I$. In this case 
	either the relation 
$\alpha c_1 $ or the relation $c_k \alpha$ yields that there is no derivation sending the loop $\alpha$ to its source, and using Lemma~\ref{rad=nonrad} we conclude that  $\HH^1_{rad}(A)=\HH^1(A)$. 

 \end{proof}

 \begin{Remark} (1) The proof of Theorem~\ref{Brauer}, shows that for all characteristics of $K$, $\HH^1_{rad}(A)$ is solvable for  a Brauer graph algebra $A$, unless $A$ is isomorphic to the trivial extension of the Kronecker algebra.

 (2)  In characteristic $2$, if $A$ is a Brauer graph algebra which has a relation $\alpha^2=0$ for a loop $\alpha$,
one can easily construct an example to show that $\HH^1_{rad}(A) \neq \HH^1(A)$. 
Consider, for instance, the algebra $KQ/I$ where $Q$ is the quiver

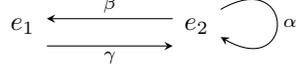
\[
\begin{tikzcd}[arrow style=tikz,>=stealth,row sep=4em]
e_1  
\arrow[rr,shift right=2ex, swap, "\gamma"]
&& e_2 \arrow[out=30,in=-30,loop,"\alpha"]
  \arrow[ll,shift right=.4ex,swap,"\beta"]
\end{tikzcd}
\]

 and $I=\langle \alpha^2, \gamma\beta, \beta\gamma\alpha - \alpha\beta\gamma \rangle$. The derivation that sends $\alpha$ to its vertex and any other arrow
to zero is well defined and does not preserve the radical. However as claimed above, one easily verifies by hand that  $\HH^1(A)$ is solvable since its second derived Lie algebra is zero.

 \end{Remark}


\subsection{Symmetric tame algebras}\label{sec:weakly symmetric}

In this subsection $A$ is a symmetric tame algebra which is in the classification of Skowro\'nski's survey paper in \cite{Skow}. 
Our aim is to prove the following.
\begin{Theorem}
\label{selfinjetivetame}
Let $A$ be a symmetric tame algebra that appears in the classification in \cite{Skow}  not derived equivalent to $K[X]/(X^r)$ 
when $\mathrm{char}(K)\mid r$  and not derived equivalent to the trivial extension of the Kronecker algebra if $\rm{char}(K) \neq 2$.  
Then $\mathrm{HH}^1(A)$ is a solvable Lie algebra. 
\end{Theorem}

These algebras include the algebras of dihedral, semi-dihedral and quaternion type as well as Brauer graph algebras. We have seen in the previous sections that for these algebras the first Hochschild cohomology as a Lie algebra is solvable, except for a small number of special cases. 

  We start by recalling the following result  from \cite{Skow} (using the notation in that paper), which provides the 
derived equivalence classes of algebras of non-simple connected symmetric algebras of finite type. 
\begin{Theorem}[\cite{Skow}]
\label{finite}
The algebras $N^{em}_e , m\geq 2, e\geq 1$,  $D(m), m\geq   2$, $T(K\Delta(A_n))$ $n\geq 1$, $T(K\Delta(D_n))$ 
$n \geq 4$, $T(K\Delta(E_n))$, $6 \leq n\leq 8$ and $D^{'}(m)$, $m \geq 2$ and $\mathrm{char}(K)=2$, form a 
complete family of representatives of the derived classes of the non-simple connected symmetric algebras of finite 
type.
\end{Theorem}

\begin{Proposition}
Let $A$ be as in Theorem \ref{finite}, excluding Nakayama algebras  $N^{em}_e$ where $e=1$ and $\mathrm{char}(K)$ divides $m+1$. The Lie 
algebra $\mathrm{HH}^1(A)$ is solvable. 
\end{Proposition}
\begin{proof}

Let $N^{em}_e , m\geq 2, e\geq 1$ be a Nakayama algebra with $e$ vertices and such that all the compositions of $em+1$ consecutive 
arrows generate  the admissible ideal. If $e=1$, then  $N^{em}_e \cong K[x]/(x^{m+1})$. If  $p$ divides $m+1$, then 
$\mathrm{HH}^1(K[x]/(x^{m+1}))$ is a perfect Lie algebra, therefore not solvable. If $p$ does not divide $m+1$, then 
$x || e$ is not in $\Sigma_0$ and since the algebra is graded by Theorem~\ref{trivialclasses} we have that  $\mathrm{HH}^1(N^{em}_e)$ is a solvable Lie algebra. If $e>1$ the statement follows from Theorem \ref{Liestronglysolvable-general}.  

The trivial extension algebra $T(K\Delta(A_n))$ for $n\geq 1$ is the Nakayama algebra $N^n_n$ therefore the 
solvability of $\mathrm{HH}^1(T(K\Delta(A_n))$ follows from the previous paragraph. For the other self-injective algebras of Dynkin type, 
we proceed as follows. 

The algebras $D(m)$ and $D'(m)$ are defined in \cite[Section 3.14]{Skow} and a direct calculation shows that $\HH^1(A)$ is solvable for both $D(m)$ and $D'(m)$. 

From Theorem \ref{Liestronglysolvable-general} we directly obtain the solvability of $\mathrm{HH}^1(T(K\Delta(D_n)))$ and of $\mathrm{HH}^1(T(K\Delta(E_n)))$.

\end{proof}

The description of symmetric algebras of Euclidean type given in \cite{Skow} is as follows.
 
\begin{Theorem}\cite[Thm. 4.16]{Skow} 
 Let $A$ be a symmetric algebra of Euclidean type. The following are equivalent
\begin{itemize}
\item $A$ is symmetric and has nonsingular Cartan matrix.
\item $A$ is derived equivalent to an algebra of the form $A(p,q)$, $\Lambda(n)$ or $\Gamma(n)$.
\end{itemize} 
\end{Theorem}

\bigskip

\begin{Proposition}\label{prop:nonsingular Euclidean}
The first Hochschild cohomology space of any symmetric algebra $A$ of Euclidean type with nonsingular Cartan matrix is a solvable Lie algebra.
\end{Proposition}
\begin{proof}
Such an algebra is a Brauer graph algebra, therefore the statement follows from Theorem \ref{Brauer} 
except for  $\Lambda(n)$ in characteristic 2. So suppose that $A = \Lambda(n)$ and that the characteristic 
of $K$ is 2. Then $A / J(A)^3$ is graded and it follows from Corollary~\ref{loops2} that  $\HH_{rad}^1(A)$ is 
solvable.

We verify by hand that $\HH^1_{rad}(A) = \HH^1(A)$. 
\end{proof}

In order to describe what happens when $A$ is a symmetric algebra of Euclidean type with singular Cartan matrix we need 
a preliminary result.
 
\begin{Proposition}[\cite{Skow}]
Let $A$ be a 
symmetric algebra of Euclidean type with singular Cartan matrix. There exists an Euclidean canonical algebra $C$ such that $A$ is isomorphic to 
the trivial extension $T(C)$. 
\end{Proposition}

\begin{Proposition}\label{prop:canonical}
Let $C= C(p,q,r)$ be a canonical algebra with parameters $p,q$ and $r$ in $\mathbb{N}_{\geq 2}$. Then $\HH^1(T(C))$ is solvable. 
\end{Proposition}

\begin{proof}
Let $A = T(C)$. The quiver of $A$ is given by the quiver of $C$, with two additional parallel arrows $\alpha$ and $\beta$,  starting at the sink vertex of $A$ and ending at the source vertex of $A$.  We begin by checking that $\HH^1(A / J(A)^3)$ is solvable. For this note that $A / J(A)^3$ is graded and that since the quiver has no loops, $\Sigma_0$ is empty. Furthermore, we check by hand that neither $\alpha || \beta$ nor $\beta || \alpha$ are in $\Sigma_1$. Therefore by Theorem~\ref{trivialclasses}, $\HH^1(A / J(A)^3)$ is solvable. By  \cite[Cor. 2.5]{ER},  we then have that $\HH^1_{rad}(A)$ is solvable and since $A$ has no loops the result follows from \cite[Prop. 2.7]{ER}. 
\end{proof}

\begin{Corollary}
The first Hochschild cohomology $\HH^1(A)$ for  any  symmetric algebra $A$ of Euclidean type with singular Cartan matrix is a solvable Lie algebra.
\end{Corollary}

\begin{proof}
It is shown in \cite{Skow} that $A$ is isomorphic to the trivial extension of an Euclidean canonical algebra of the form $C(2,3,3), C(2,3,4), C(2,3,5)$ or  $C(2,2,r)$ with $r\geq 2$. 
\end{proof}

Next we consider weakly symmetric  algebras of tubular type. We first recall the following two results.

\begin{Theorem}[\cite{Skow}]
\label{nondomensticpoly1}
Let $A$ be a standard weakly symmetric algebra of tubular type with non-singular Cartan matrix. Then $A$ is derived equivalent to an algebra of the form $A_1(\lambda)$, $A_2(\lambda)$ with $\lambda\in K\ \{0,1\}$, $A_3$,$A_4$,
$A_5$ or $A_{12}$.
\end{Theorem}

\begin{Theorem}[\cite{Skow}]
\label{nondomensticpoly2}
Let $A$ be a non-standard non-domestic weakly symmetric algebra of polynomial growth with non-singular Cartan matrix. Then $A$ is derived equivalent to an algebra of the form $\Lambda_1$, $\Lambda_3(\lambda)$ where $\lambda$  $\in K\ \{0,1\}$, $\Lambda_4$ or $\Lambda_9$.
\end{Theorem}

\begin{Proposition}
Let $A$ be an algebra either as in Theorem \ref{nondomensticpoly1} or as in Theorem \ref{nondomensticpoly2}. The Lie algebra 
$\mathrm{HH}^1(A)$ is solvable.
\end{Proposition}

\begin{proof}
For the algebras $A_1(\lambda)$,$A_3$,$A_4$, $A_{12}$, $\Lambda_4$ and $\Lambda_9$ the solvability of the first Hochschild cohomology space
follows from Theorem \ref{Liestronglysolvable-general}.
For  $A_5$, $\Lambda_1$, and $\Lambda_3(\lambda)$ we directly verify that $\Sigma_0$ is empty and the result follows from Theorem~\ref{trivialclasses-nongraded}. For $A_2(\lambda)$ we check that if the characteristic of $K$ is not 2, $\Sigma_0$ is empty and the result follows from Theorem~\ref{trivialclasses} and if characteristic is 2 then the result follows from Corollary~\ref{loops2}. 
\end{proof}

\begin{Theorem}\label{thm:domestic infinite}\cite[Thm. 4.14]{Skow}
Let $A$ be a non-standard domestic self-injective algebra of infinite representation type. Then $A$ is derived equivalent to an algebra of the form $\Omega(n)$, for $n \geq 1$. 
\end{Theorem}

\begin{Proposition}
Let $A$ be an algebra as in Theorem~\ref{thm:domestic infinite}. Then $HH^1(A)$ is solvable.  
\end{Proposition}

\begin{proof}
The algebras $\Omega(n)$ have the same quiver as a Brauer graph algebra with one loop and one cycle of length at least 2. An argument similar to that for Brauer graph algebras in the proof of Theorem~\ref{Brauer} and in characteristic 2 similar to the proof of Proposition~\ref{prop:nonsingular Euclidean} give the result. 
\end{proof}

The case where $A$ has singular Cartan matrix is covered by the following results. 
Let us recall the following theorem.

\begin{Theorem}[\cite{Skow}]
\label{tubularsingularCartan}
Let $A$ be a self-injective algebra. The following statements are equivalent:
\begin{itemize}
\item $A$ is symmetric of tubular type and has singular Cartan matrix.
\item $A$ is derived equivalent to the trivial extension of a canonical tubular algebra.
\end{itemize}
\end{Theorem}

\begin{Proposition}
Let $A$ be derived equivalent to the trivial extension of a canonical tubular algebra. Then $\mathrm{HH}^1(A)$ is solvable.
\end{Proposition}


\begin{proof} If $A$ is symmetric of tubular type with singular Cartan matrix then $A$ is derived equivalent to the trivial extension of a canonical algebra of type $C(2,4,4)$, $C(3,3,3)$, $C(2,3,6)$ or $C(2,2,2,2,\lambda)$ for $\lambda \in K\ \{0,1\}$. For the algebras of the form $C(p,q,r)$, the result directly follows from Proposition~\ref{prop:canonical}. 
For $C= C(2,2,2,2,\lambda)$, we note that $T(C)$ is graded and its quiver has no loops, and it has two  parallel arrows $\alpha$ and $\beta$ from the sink of 
the quiver of $C$ to its source. We check that neither $\alpha || \beta$ nor $\beta || \alpha$ are in $\Sigma_1$. Therefore the result follows from Theorem~\ref{trivialclasses}.
\end{proof}

The structure of arbitrary standard self-injective algebras of polynomial growth is described 
by the following theorem.

\begin{Theorem}[\cite{Skow}]
Let $A$ be a nonsimple basic connected self-injective algebra. The algebra $A$ is standard of polynomial growth if and only if $A$ is 
isomorphic to a self-injective algebra of Dynkin type, Euclidean type or tubular type. 
\end{Theorem}




\subsection{Quantum complete intersections}

The Hochschild cohomology of quantum complete intersections has been extensively studied, see for example \cite{BE, BKL, EH, O}. 

 Recall that $KQ/I$ is a quantum complete intersection of rank $r$, if
$$ KQ/I = K \langle X_1, \ldots, X_r \rangle / \langle X_j X_i - q_{ij} X_i X_j, X_i^{{n_i}}, \mbox{ for } 1 \leq i< j \leq r,  \rangle $$
where $n_i \in \mathbb{N}_{\ge2}$ for all $i$  and $q_{ij}\neq 0$ for all $i,j$. 

In this section we apply Proposition~\ref{localB0} 
to calculate the  Lie algebra structure of the first Hochschild cohomology of those quantum complete intersections  of rank $r$ with 
parameters different from $0$ and $1$, that is, quantum complete intersections with non commuting variables. 
In particular, it follows from our results that to a certain degree this Lie algebra structure is independent of the rank and the parameters.

\begin{Proposition} 
 Let $A$ be a quantum complete intersection of rank $r$ with $q_{ij}\neq 1$ for all $i<j$.
Then $\mathrm{HH}^1(A)$ is a strongly solvable Lie algebra and so, in particular, it is solvable. 
\end{Proposition}

\begin{proof}
We   denote the relations as follows:  $\rho_{ij}$ is the relation $X_jX_i -q_{ij}X_iX_j$
 and $\rho_{ii}$ is the relation  $X_i^{n_i}$, for $1 \leq i < j \leq r$. 
The first step in order to apply Proposition~\ref{localB0} is to show that  $\cL_0=0$.
 One of the terms appearing in the  image of the differential  of $X_i || e_0$ is $(1-q_{ii+1})\rho_{ii+1} || X_{i+1}$. 
As a consequence, $X_i || e_0$ is never in  
the kernel since the only other element whose image under the differential contains $\rho_{ii+1}$ is $X_{i+1} || e_0$, 
but in this case we get  $(1- q_{ii+1})\rho_{ii+1} || X_{i}$. Since $X_i \neq X_{i+1}$ the previous statement follows. 
Now, we just have to check that $X_i||X_j$ is not in $\Sigma_1$ for $i \neq j$.
For this, notice that one of the nonzero terms  in $\delta^1(X_i || X_j)$ is $\rho_{ij} || (1- q_{ij})X_j^2$ or 
$\rho_{ji} || (1- q_{ji})X_j^2$, depending on whether $i<j$ or $i>j$. Again, $\rho_{ij}$ - respectively - $\rho_{ji}$ appears again only in $\delta^1(X_j || X_i)$, 
but, as before, neither occurrence cancels the other one out. It follows from the proof of Proposition~\ref{localB0} that
  $\mathrm{HH}^1(A)$ is a strongly solvable Lie algebra.  
\end{proof}

We also consider quantum complete intersections where
we allow some of the quantum parameters to be $1$.  For convenience, we set $q_{ii} =1$ for every $i \in \{1, \dots, r\}$ and for any $i, j\in \{1, \dots, r\}$ we set
$q_{ij} = q^{-1}_{ji}$ for $1\leq j < i \leq r$. 

\begin{Proposition} 
 Let $A$ be a quantum complete intersection 
 of rank $r$ over a field of characteristic $p$ such that $n_i = p^{m_i}$ for $m_i \in    \mathbb{N}_{\ge1}$. If there exists $i \in \{1,\dots , r\}$ 
 such that  for  every $j \in \{1, \dots, r\}$ we have  $q_{ij}=1$,
 then $\mathrm{HH}^1(A)$ is not solvable.
\end{Proposition}

\begin{proof}
Assume  there exists $i \in \{1,\dots , r\}$ 
 such that for  every $j \in \{1, \dots, r\}$ we have $q_{ij}=1$.
Then we show that $\mathrm{HH}^1(A)$ contains a perfect Lie algebra
 and consequently $\mathrm{HH}^1(A)$ is not solvable. 
 We first prove that $X_i || e$ belongs to  $\ker(\delta^1)$. 
 In fact, the image of $X_i || e$ under $\delta^1$ is 
 $\sum_{j=1}^{i-1} \rho_{ji}|| (X_j-q_{ji}X_j)+$
 $\rho_{ii}|| n_i  X_{i}^{n_i-1}+$
 $\sum_{j=i+1}^c \rho_{ij}|| (X_j-q_{ij}X_j)$. The first and the last sums are
 zero since 
 $q_{ij}=1$ for every $j\in \{1, \dots, r\}$.
The second expression is zero because $char (K)$ divides
 $n_i$. Therefore $X_i|| e \in \ker(\delta^1)$. Since for a fixed $i$
 all the $q_{ij}=1$,  
  it is easy to show that for 
 every $h \in \{1, \dots, p^{m_i}-1\}$ we have that $x_i ||\ x_i^h$ is
 an element of $\ker (\delta^1)$. If we denote by $X^0_i:=e$, then we have that the Lie subalgebra $\mathcal{S}$ of $\mathrm{HH}^1(A)$ with $K$-basis 
 $X_i || X^h_i$ for $h \in \{0, \dots ,p^{m_i}-1\}$ is perfect and therefore $\mathrm{HH}^1(A)$  is not
 solvable. 
 
\end{proof}


\subsection{Examples of algebras with nonsolvable first Hochschild cohomology}

In this section we give two  examples of  families of algebras for which the Lie algebra given by the first Hochschild cohomology is not solvable but rather semi-simple. They are a generalisation of  the example  of the Kronecker algebra given in \cite{CSS}, see also \cite{ER}.  In characteristic zero, the Lie
structure of the first cohomology space of both families can be deduced from \cite{SF}.

These families consist of monomial algebras with radical square zero over a field of arbitrary characteristic.  For the first one, let
 $n, m \in \mathbb{N}_{\geq 1}$. 
We denote by $Q_{n,m}$ the quiver  with $n$ vertices $1, \ldots, n$  and  a set of $m$ parallel arrows denoted by 
$\alpha_{i,1}, \dots \alpha_{i,m}$ from $i$ to $i+1$ for each pair of consecutive vertices 
$(i, i+1)$.
For $n=2$, the quiver $Q_{2,m}$ is the $m$-Kronecker quiver. Set $A_{n,m} = KQ_{n, m } / J(KQ_{n,m })^2$.

\begin{Proposition}
\label{Non-solv1}
There is an isomorphism of Lie algebras $$\mathrm{HH}^1(A_{n,m})\cong \prod_{i=1}^n \mathfrak{sl}_{m}(K),$$ for $n, m \ge 2$. 
In particular,  $\mathrm{HH}^1(A_{n,m})$ is a solvable Lie algebra if and only if $\mathrm{char}(K)=2$ and $m=2$. 
\end{Proposition}
  
\begin{proof}
A short calculation shows that, for $j \neq l$  where $j,l\in \{1, \dots m\}$, $h \in \{1, \dots m-1\}$ and $i \in \{1, \dots n\}$, the set
$\{\alpha_{i,h}||\alpha_{i,h} - \alpha_{i,h+1}||\alpha_{i,h+1},  \alpha_{i,j} || \alpha_{i,l} \}$  is a basis of $\mathrm{HH}^1(A_{n,m})$.  
Since the Lie 
bracket of  elements corresponding to parallel arrows starting at different vertices is zero, the Lie algebra structure of $\mathrm{HH}^1(A_{n,m})$ decomposes 
as a product of Lie algebras. The next step is to show that for a fixed $i$ each of these Lie algebras is 
isomorphic to $\mathfrak{sl}_{m}(K)$. We consider the  basis for $\mathfrak{sl}_{m}(K)$ given by the set of elementary matrices $\{e_{st}\}_{s,t}$ for $s \neq t$, together 
with $h_s = e_{ss} - e_{s+1,s+1}$.  
 For $i$ fixed,  we denote $\alpha_{i,j}$ by $\alpha_j$. 
Then if we write $e_{st}:=\alpha_s || \alpha_t$ and $h_s:= \alpha_s || \alpha_s - \alpha_{s+1} || \alpha_{s+1}$ it is easy to show that the above statement follows. 
Therefore $\mathrm{HH}^1(A)$ is isomorphic to $\prod_{i=1}^n \mathfrak{sl}_{m}(K)$.  Since $\mathfrak{sl}_m(K)$ is solvable only for $m=2$ in $\mathrm{char}(K)=2$,   the Lie algebra obtained is solvable if and only if
$\mathrm{char}(K)=2$ 
and $m=2$.
\end{proof}

An analogous proof to the above shows the following generalisation of Proposition~\ref{Non-solv1}. Let $n, {\mathbf m}$ be such that  
$\mathbf{m} = (m_1, \ldots, m_k)$   and $ n, m_1, \ldots, m_k \in \mathbb{N}_{\geq 1}$.
We denote by $Q_{n,{\mathbf m }}$ the quiver  with $n$ vertices $1, \ldots, n$  and  a set of $m_i$ parallel arrows denoted by 
$\alpha_{i,1}, \dots \alpha_{i,m_i}$ from $i$ to $i+1$ for each pair of consecutive vertices 
$(i, i+1)$. 
For $n=2$ and ${\mathbf m} = m$, the quiver $Q_{2,m}$ is the $m$-Kronecker quiver. Set $A_{n,{\mathbf m }} = KQ_{n,{\mathbf m }} / J(KQ_{n,{\mathbf m }})^2$.

\begin{Corollary} For $A_{n,{\mathbf m }}$ as above, we have
$$ \mathrm{HH}^1(A_{n,{\mathbf m }})\cong \prod_{i=1}^n \mathfrak{sl}_{m_i}(K),$$ for $n, m_1, \ldots, m_k \ge 2$. 
\end{Corollary}

It can be shown in a similar way that $\mathrm{HH}^1(K Q_{n,{\mathbf m }} )
\cong \prod_{i=1}^n \mathfrak{sl}_{m_i}(K),$ for $n, m_1, \ldots, m_k \ge 2$.

\begin{Corollary}
\label{Non-solv2}
Let $A = KQ/I$ be a  finite dimensional algebra such that $Q$ contains $Q_{n,m}$
 as a subquiver and the other arrows of $Q$ form a simple directed graph. Suppose that $I$ contains all or none of the relations of degree 
$2$ involving the arrows of $Q_{n,m}$  and that for any arrow $\alpha$ in $Q_{n,m}$ and any arrow $\beta$ 
not in $Q_{n,m}$, either $\alpha\beta\notin I$ or $\beta\alpha \notin I$.  
Then $$\mathrm{HH}^1(A)/\mathrm{rad}(\mathrm{HH}^1(A))\cong \prod_{i=1}^n \mathfrak{sl}_{m}(K).$$
\end{Corollary}

\begin{proof}
By construction, $\mathrm{HH}^1(A)= \prod_{i=1}^n \mathfrak{sl}_{m}(K) \oplus \mathcal{S}$, where 
$\mathcal{S}$ is a solvable Lie algebra. The statement follows.   
\end{proof}

It follows from Corollary~\ref{Non-solv2}, see also \cite{ER} in characteristic zero, that the first Hochschild cohomology of a special 
biserial algebra $A$ (beyond the Kronecker algebra) is not necessarily solvable. This is the case, for example, when the quiver contains 
$Q_{2,m}$ for some $m$ as subquiver with all relations of length two for all arrows in $Q_{2,m}$. In that case $\mathfrak{sl}_{2}(K)$ is a Lie subalgebra of $\mathrm{HH}^1(A)$.

\end{document}